\documentclass[12pt]{article}
\usepackage{amsfonts,amsmath,amsthm, amssymb}
\usepackage{latexsym, euscript, epic, eepic}
\newtheorem{lemma}{Lemma}[section]
\newtheorem{theorem}{Theorem}[section]
\newtheorem{definition}{Definition}[section]
\newtheorem{proposition}{Proposition}[section]

\newtheorem{remark}{Remark}

\def \a{\alpha}

\def \e{\epsilon}

\def \w{\omega}
\def \Om{\Omega}
\def \om{\omega}

\begin{document}
\title{The Hausdorff dimension of average conformal repellers under random perturbation}
\author{ Yun Zhao$^{\dag}$\ Yongluo Cao$^{\dagger}$\ Jungchao Ban$^{\ddag^{a,b}}$ \\
\small \it $\dag$ Department of mathematics\\
\small \it Suzhou University\\
\small \it Suzhou 215006, Jiangsu, P.R.China\\
\small \it zhaoyun@suda.edu.cn, ylcao@suda.edu.cn\\
\small\it $\ddag^{a}$ Department of mathematics\\
\small\it National Hualien University of Education\\
\small \it Hualien 97003, Taiwan\\
\small \it $\ddag^{b}$ Taida Institute for Mathematical Sciences\\
\small \it  National Taiwan University\\
\small\it Taipei 10617, Taiwan\\
 \small\it jcban@mail.nhlue.edu.tw }
\date{}
 \footnotetext{Partially supported by   NSFC(10571130), NCET, and SRFDP of China. }
\footnotetext{2000 {\it Mathematics Subject classification}:
37C45, 28A78, 34D10.} \maketitle

\begin{center}
\begin{minipage}{120mm}
{\small {\bf Abstract.} We prove that the Hausdorff dimension of
an average  conformal repeller is stable under random
perturbations. Our perturbation model uses the notion of a bundle
random dynamical system. }
\end{minipage}
\end{center}

\vskip1cm

{\small{\bf Key words and phrases} \ Hausdorff dimension,
topological pressure, random dynamical system} \vskip1cm

\section{Introduction.}
In the dimension theory of dynamical system, only the Hausdorff
dimension of invariant sets of conformal dynamical system is well
understood. Since the work of Bowen, who was the first to express
the Hausdorff dimension of an invariant set as a solution of an
equation involving topological pressure. Ruelle \cite{ruelle}
refined Bowen's method and get the following result. If $J$ is a
mixing repeller for a $C^{1+\a}$ conformal expanding map $f$ on a
Riemannian manifold $M$, then the Hausdorff dimension of $J$ can
be obtained as the zero $t_0$ of $t\mapsto \pi_f(-t\log || D_x
f||)$, where $\pi_f$ denotes the topological pressure functional.
This statement is known as the Bowen-Ruelle formula, and we
sometimes call the equation involving topological pressure Bowen
equation. And Gatzouras and Peres relaxed the smoothness
$C^{1+\a}$ to $C^1$ in \cite{gp}.

Recently, different version of topological pressure has become an
useful tool in calculating the Hausdorff dimension of a
non-conformal repeller. For $C^1$ non-conformal repellers, Zhang
used singular values of the derivative $D_xf^n$ for all $n\in
\mathbb{N}$, to define a new equation which involves the limit of
a sequence of topological pressure, then he showed that the upper
bound of the Hausdorff dimension of repeller was given by the
unique solution of the equation, see \cite{zhang} for details.
Barreira considered the same problem in \cite{barreira1}. By using
the non-additive thermodynamic formalism which was introduced in
\cite{barreira2} and singular value of the derivative $D_xf^n$ for
all $n\in \mathbb{N}$, he gave an upper bound of box dimension of
repeller under the additional assumptions that the map was
$C^{1+\a}$ and $\a$-bunched. This automatically implies that for
Hausdorff dimension. In \cite{cao1}, by using the sub-additive
topological pressure which was studied in \cite{cao2}, the author
proved that the upper bound of Hausdorff dimension for $C^1$
non-conformal repellers obtained in \cite{barreira1,falconer,
zhang} were same and it was the unique root of Bowen equation for
sub-additive topological pressure, we point out that the map is
only need to be $C^1$ without any additional condition in
\cite{cao1}.

In \cite{ban}, the authors introduced the notion of
average\underline{} conformal repeller in the deterministic
dynamic systems which was a generalization of quasi-conformal and
asymptotically conformal repeller in \cite{barreira2,pesin}, and
they proved that the Hausdorff dimension and box dimension of
average conformal repellers was the unique root of Bowen equation
for sub-additive topological pressure. In that paper, the map is
only needed $C^1$, without any additional condition.

For random repellers, Kifer proved that the Hausdorff dimension of
a measurable random conformal repeller was the root of the Bowen
equation which can be seen as a random version of the
deterministic case, see \cite{kifer} for details. And in
\cite{bog}, the authors generalized this result to
almost-conformal case. In \cite{zhao1}, using the idea in the
deterministic case \cite{ban},  authors introduced the notion of
random average conformal repeller, and they proved that the
Hausdorff dimension of random average conformal repellers was the
unique root of Bowen equation for random sub-additive topological
pressure which was studied in \cite{zhao2}.

Motivated by the work in \cite{bog}, where the authors showed that
the Hausdorff dimension of the conformal repeller was stable under
random perturbation, we consider a random perturbation of the
deterministic average conformal repeller which is modeled using
the notion of a bundle random dynamical system(RDS for short).
Namely, let $\vartheta$ be an ergodic invertible transformation of
a Lebesgue space $(\Omega, \mathcal{W}, {\mathbb{P}})$ and
consider a measurable family $T=\{T(\om): M\rightarrow M \} $ of
$C^{1+\a}$ maps, that is to say, $(\om, x)\mapsto T(\om)x$ is
assumed to be measurable. This determines a differentiable RDS via
$T(n,\om):= T(\vartheta^{n-1}\om)\circ\cdots \circ
T(\vartheta\om)\circ T(\om)(n\in {\mathbb{N}}) $. Further, Let
$E\subset \Om \times M$ be a measurable set such that all
$\om$-sections $E_\om:=\{ x\in M : (\om,x)\in E\}$ are compact. If
$\mathcal{K}$ denotes the collection of all compact subsets of $M$
endowed with the Hausdorff topology, this is equivalent to saying
that $\mathcal{K}$-valued multifunction $\om\mapsto E_\om$ is
measurable. Here and in what follows we think of $E_\om$ being
equipped with the trace topology, i.e. an open set $A\subset
E_\om$ is of the form $A = B \cap E_\om$ with some open set
$B\subset M$. We call $E$ is $T$-invariant if $T(\om) E_\om =
E_{\vartheta \om}\ \mathbb{P}-a.s.$, and in this situation the
Hausdorff dimension of the fiber $E_\om$ is a $\mathbb{P}$-a.s.
constant, see \cite{cf}. The map $\Theta: E \rightarrow E $ is
defined by $\Theta(\om, x) = (\vartheta \om, T(\om)x)$, and we
call it the skew product transformation.

The aim of this paper is to make rigorous the statement that if a
bundle RDS is close to an average conformal expanding map on a
repeller then the corresponding Hausdorff dimension are close.

The paper is organized as follows. In section 2, we will recall
the main result in \cite{ban}. In section 3, we introduce some
random notions and our model of random perturbation, we point out
that this was essentially inspired by a remarkable result of Liu
\cite{liu}. In section 4, we formulate and prove our main result
which says that the Hausdorff dimension of an average  conformal
repeller is stochastically stable.

\section{Dimension of average conformal repeller}

In this section, we will recall the notion of sub-additive
topological pressure and the main result in \cite{ban} which says
that the Hausdorff dimension of an average conformal repeller can
be given by the unique root of the sub-additive topological
pressure. Moreover, we will give some preliminary results.

Let $f:X\rightarrow X$ be a continuous map on a compact  space $X$
with metric $d$. A subset $E\subset X$ is called an
$(n,\e)$-separated set with respect to $f$ if $x\neq y\in E$
implies $d_n(x,y):=\max_{0\leq i\leq n-1}d(f^ix,f^iy)>\e$. Let
${\mathcal{F}}=\{ \phi_n\}_{n\geq 1}$ denote a sub-additive
potential on $X$, that is to say $\phi_n:X\rightarrow \mathbb{R}$
is continuous for each $n\in \mathbb{N}$ and satisfying
$$
\phi_{n+m}(x)\leq \phi_n(x)+\phi_{m}(f^n(x)),\ \forall
n,m\in\mathbb{N},x\in X.
$$
Following the way in \cite{cao2}, we define the sub-additive
topological pressure
$$
\pi_f(\mathcal{F},n,\e)=\sup \left\{ \sum_{x\in E} \exp \phi_n(x)
: E \ is\ an\ (n,\e)-separated \ subset \ of\ X\right\}
$$
and then call
$$
\pi_f(\mathcal{F})=\lim_{\e\rightarrow 0} \limsup_{n\rightarrow
\infty} \frac{1}{n} \log \pi_f(\mathcal{F},n,\e)
$$
the sub-additive topological pressure of $\mathcal{F}$ with
respect to $f$. If there is no confusion caused, we simply call
$\pi_f(\mathcal{F})$ the sub-additive topological pressure of
$\mathcal{F}$.

\begin{remark} \label{zhu2} \rm
\it (1)When the continuous potential $\mathcal{F}=\{\phi_n\}$ on
$X$ is additive, i.e. $\phi_n(x)=\sum_{i=0}^{n-1}\phi (f^ix)$ for
some continuous function $\phi:X\rightarrow \mathbb{R}$, then
$\pi_f(\mathcal{F})$ is the classical topological pressure, see
\cite{walters} for details, and we denote it simply by
$\pi_f(\phi)$; (2)When the continuous potential
$\mathcal{F}=\{\phi_n\}$ on $X$ is sup-additive, that is to say,
$\phi_{n+m}(x)\geq \phi_{n}(x)+\phi_m(f^nx),\ \forall
n,m\in\mathbb{N},x\in X$, we also can define the sup-additive
topological pressure. And the pressures are equal under some
special case, see \cite{ban}.
\end{remark}

Let $\mathcal{M}(X,f)$ denote the space of all $f$-invariant Borel
probability measures and $\mathcal{E}(X,f)$ denote the subset of
$\mathcal{M}(X,f)$ with ergodic measures. For $\mu\in
\mathcal{M}(X,f)$, let $h_{\mu}(f)$ denote the measure-theoretic
entropy of $f$ with respect to $\mu$, and let
$\mathcal{F}_{*}(\mu)$ denote the following limit
$$
\mathcal{F}_{*}(\mu)=\lim_{n\rightarrow\infty} \frac{1}{n} \int
\phi_n {\mathrm{d}}\mu.
$$
The relation between $\pi_f(\mathcal{F}),h_{\mu}(f)$ and
$\mathcal{F}_{*}(\mu)$ is given by the following variational
principle which is proved in \cite{cao2}, and the random version
of the following theorem is proved in \cite{zhao2}.

\begin{theorem}[Variational  principle] \label{dl21} \rm
\it Let $\mathcal{F}$ be a sub-additive potentials on a compact
metric space $X$, and $f:X\rightarrow X$ is a continuous
transformation, then
 $$   \pi_f({\mathcal{F}})=\left\{
     \begin{array}{cc}
     &-\infty, ~~~~~~~~~~~~~~~~~~~~~~~ if\ {\mathcal{F}}_{*}(\mu)=-\infty\ for\ all\ \mu \in
     \mathcal{M}(X,f)\\
     &\sup \{ h_{\mu}(f) +{\mathcal{F}}_{*}(\mu) : \mu \in
     \mathcal{M}(X,f),{\mathcal{F}}_{*}(\mu)\neq -\infty \},\  otherwise.
     \end{array}
     \right.
$$
 \end{theorem}

\begin{proposition} \label{jiamt21} \rm
 \it let $f_i: X_i \rightarrow X_i(i=1,2)$ be a continuous map of a
 compact metric space $(X_i,d_i)$, and $\mathcal{F}=\{ \phi_n\}$ is a sub-additive potential on $X_2$. If $\varphi: X_1\rightarrow
 X_2$ is a surjective continuous map with $\varphi \circ
 f_1=f_2\circ
 \varphi$ then $\pi_{f_2}(\mathcal{F})\leq \pi_{f_1}(\mathcal{F\circ
 \varphi})$, and if $\varphi$ is a homeomorphism then $\pi_{f_2}(\mathcal{F})=\pi_{f_1}(\mathcal{F\circ
 \varphi})$, where $\mathcal{F\circ
 \varphi}=\{ \phi_n\circ \varphi \}$.
 \end{proposition}
 \begin{proof}
 We first check that $\mathcal{F\circ
 \varphi}$ is indeed a sub-additive potential on the compact
 metric space $X_1$. In fact
 \begin{eqnarray*}
\phi_{n+m}\circ \varphi (x)\leq \phi_n (\varphi x) +
\phi_{m}(f_{2}^{n}\varphi x)=\phi_n\circ \varphi (x) +\phi_m \circ
\varphi(f_{1}^{n}x)
 \end{eqnarray*}
 the equality follows from the fact that $\varphi \circ
 f_1=f_2\circ
 \varphi$.

 Let $\e >0$ and choose $\delta >0$ such that
 $d_2(\varphi(x),\varphi(y))>\e$ implies $d_{1}(x,y)>\delta$, this
 fact follows from the uniform
 continuity of $\varphi$.

 Let $E$ be an $(n,\e)$-separated set with respect to $f_2$. Since $\varphi$ is surjective, there exists a subset $F\subset
 X_1$ so that $\varphi |_F: F\rightarrow E$ is a bijection. It
 follows from the above observation that $F$ is an
 $(n,\delta)$-separated set with respect to $f_1$.
  Hence, we have
 $$
 \begin{array}{ll}
 \pi_{f_2}(\mathcal{F},n,\e)&=\sup \{ \sum_{x\in E} \exp
 \phi_n( x): E \mbox{ is an } (n,\e)\mbox{-separated subset of } X_2\}\\
&=\sup \{ \sum_{y\in F} \exp
 \phi_n(\varphi y):  E \mbox{ is an } (n,\e)\mbox{-separated subset of } X_2 \\
 &\qquad \mbox{and }\varphi |_F: F\rightarrow E \mbox{ is a bijection}\}\\
&\leq  \sup \{ \sum_{y\in F} \exp
 \phi_n(\varphi y): F \mbox{ is an } (n,\delta)\mbox{-separated subset of } X_1\}\\
&=\pi_{f_1}(\mathcal{F}\circ \varphi,n,\delta)
 \end{array}
 $$
Since $\e\rightarrow 0$ then $\delta\rightarrow 0$, then we can
have
 $$
\pi_{f_2}(\mathcal{F})\leq \pi_{f_1}(\mathcal{F\circ
 \varphi})
 $$

 If $\varphi$ is a homeomorphism then we can apply the above with
 $f_1,f_2,\varphi, \mathcal{F}$ replaced by $f_2, f_1, \varphi^{-1}, \mathcal{F}\circ
 \varphi$ respectively to give $\pi_{f_2}(\mathcal{F})\geq \pi_{f_1}(\mathcal{F\circ
 \varphi})$. Thus  the proof is finished.
 \end{proof}

 \begin{proposition} \label{jiamt22} \rm
 \it Let $f:X\rightarrow X$ be a continuous map on a compact metric
space, and $ \phi : X\rightarrow \mathbb{R}$ is a continuous
function  on $X$. Suppose $ \varphi_{\e}: X\rightarrow \mathbb{R}$
is a continuous function on $X$ for every $\e>0$ and
$\lim_{\e\rightarrow 0}\varphi_{\e}=\phi$, then
$$
\lim_{\e\rightarrow 0} \pi_{f}(\varphi_{\e})=\pi_{f}(\phi).
$$
 \end{proposition}
 \begin{proof} This immediately follows from the continuity of the
 classical topological pressure.\end{proof}

 Now we introduce the definition of average conformal repeller.
 And
 the dimension of the repeller can be obtained by the unique root
 of the corresponding  sub-additive topological pressure.

 Let $M$ be a $C^{\infty}$ $m$-dimensional  Riemannian manifold.
 Let $U$ be an open subset of $M$ and $f:U\rightarrow M$ be a
 $C^1$ map. Suppose $J\subset U$ is a compact $f$-invariant
 subset. Let $\mathcal{M}(f|_J)$, $\mathcal{E}(f|_J)$ denote
 the set of all $f$-invariant measures and the set of all ergodic
 invariant measures supported on $J$ respectively. For any $\mu\in
 \mathcal{E}(f|_J)$, by the Oseledec multiplicative ergodic
 theorem (see \cite{osel}), we can define Lyapunov exponents $\lambda_1(\mu)\leq \lambda_2(\mu)\leq \cdots \leq
 \lambda_m(\mu)$, $m=\mathrm{dim}M$.

 \begin{definition} \label{dy21} \rm
 \it A compact $f$-invariant set $J$ is called an average conformal repeller for $f$ if
 for any $\mu\in \mathcal{E}(f|_J)$, $\lambda_1(\mu)= \lambda_2(\mu)= \cdots
 =\lambda_m(\mu)>0$.
 \end{definition}

 \begin{remark} \label{zhu1} \rm
 \it We point out that if $J$ is an average conformal
 repeller for $f$, it is indeed a repeller in the usual  way(see
 \cite{cao3}) that: $\exists \lambda>1, C>0$ such that for all $x\in J$ and $v\in T_xM$
 $$
||D_xf^n(v)||\geq C\lambda^n||v||,~~~\forall n\geq 1.
 $$
 \end{remark}

 \begin{proposition} \label{jiamt23} \rm
 \it If $J$ is an average conformal repeller for $f$, then
 $$
 \pi_f(\Phi)=\lim_{n\rightarrow \infty} \frac{1}{n}\pi_{f^n}(-\log \| D_xf^n\|)
 $$
 where $\Phi=\{-\log \| D_xf^n\|\}_{n\geq 1}$ is a
 sup-additive potential and $\pi_f(\Phi)$, $\pi_{f^n}(-\log \| D_xf^n\|)$ denote the sup-additive
topological pressure of $\Phi$ with respect to $f$, classical
topological pressure
 of $-\log \| D_xf^n\|$ with respect to $f^n$ respectively.
 \end{proposition}
 \begin{proof}
 Let $\Psi=\{ -\log m(D_xf^n)\}_{n\geq 1}$ denotes the
 sub-additive potential. First note that by the definition of
 topological pressure, we have
 $$
 \frac{1}{k} \pi_{f^k}(-\log\| D_xf^k\|) \leq
 \frac{1}{k}\pi_{f^k}(-\log m(D_xf^k)),\ \forall k\geq 1.
 $$
 And since $J$ is an average conformal repeller, the  measure-theoretic entropy
 map $\mu\mapsto h_{\mu}(f)$ is upper-semi-continuous by remark \ref{zhu1}. By
 proposition 2.2 in \cite{cao1}, we have
 $$
\lim_{k\rightarrow\infty}\frac{1}{k}\pi_{f^k}(-\log
m(D_xf^k))=\pi_{f}(\Psi)
 $$
 where $\pi_{f}(\Psi)$ denotes the sub-additive topological
 pressure of $\Psi$ with respect to $f$. Thus we have
 \begin{eqnarray}\label{jiads21}
\limsup_{k\rightarrow\infty}\frac{1}{k} \pi_{f^k}(-\log\|
D_xf^k\|)\leq \pi_{f}(\Psi)=\pi_f(\Phi),
 \end{eqnarray}
where the last equality is proved in \cite{ban} since $J$ is an
average conformal  repeller for $f$.

 On the other hand, for any $\mu\in \mathcal{M}(f|_J)\subset
 \mathcal{M}(f^k|_J)$, we have
 \begin{eqnarray*}
 h_{\mu}(f)+\Phi_{*}(\mu)&=&\lim_{k\rightarrow\infty}\frac{1}{k}(h_{\mu}(f^k)+\int -\log
\| D_xf^k \|{\mathrm{d}}\mu)\\
&\leq&\liminf_{k\rightarrow\infty}\frac{1}{k}\pi_{f^k}(-\log \|
D_xf^k \|),
 \end{eqnarray*}
 the last inequality is follows from the classical variational
 principle for additive topological pressure of $-\log \|
D_xf^k \|$ with respect to $f^k$, see \cite{walters}. Again
because $J$ is an average conformal repeller, by the
 variational principle for the sup-additive topological pressure(see
 \cite{ban}), we have
 \begin{eqnarray} \label{jiads22}
\pi_f(\Phi)\leq
\liminf_{k\rightarrow\infty}\frac{1}{k}\pi_{f^k}(-\log \| D_xf^k
\|).
 \end{eqnarray}
 Thus the desired result immediately follows from (\ref{jiads21})
 and (\ref{jiads22}).
 \end{proof}

The dimension of an average conformal repeller can be given by the
following theorem in \cite{ban}.

\begin{theorem} \label{dl22} \rm
\it Let $f$ be $C^1$ dynamical system and $J$ be an average
conformal repeller for $f$, then the Hausdorff dimension of $J$ is
zero of $t\mapsto \pi_f(-t\Psi)$, where $\Psi=\{ \log m(D_xf^n):
x\in J, n\in \mathbb{N}\}$ and $m(A)= ||A^{-1}||^{-1}$.
\end{theorem}

\section{Random notations}
In this section, we will give some random notions and some
well-known results. Firstly, let $(\Om,\mathcal{W},\mathbb{P})$
and $\vartheta,E,T$ be described in section 1, and let
$\mathcal{M}_{\mathbb{P}}^{1}(E,T)$ denote the space of
$\Theta$-invariant measures with marginal $\mathbb{P}$ on $\Om$ of
the RDS, $\mathcal{E}_{\mathbb{P}}^{1}(E,T)$ denote the subset of
$\mathcal{M}_{\mathbb{P}}^{1}(E,T)$ with ergodic measures of the
RDS.

Let $L_{E}^{1}(\Omega,C(M))$ denote the collection of all
integrable random continuous functions on fibers, i.e. a
measurable $f: E\rightarrow \mathbb{R}$ is a member of
$L_{E}^{1}(\Omega,C(M))$ if $f(\omega): E_{\omega} \rightarrow
\mathbb{R}$ is continuous and $\| f \|_{1}:=\int \| f(\omega) \|
\mathrm{d} \mathbb{P}(\omega)<\infty$, where $\| f(\omega)
\|=\sup_{x\in E_{\omega}} | f(\omega,x)|$. If we identify $f$ and
$g$ provided $\| f-g \|_{1}=0$, then $L_{E}^{1}(\Omega,C(M))$
becomes a Banach space with the norm $\| \cdot \|_{1}$. A family
$\Phi$ = $\{ \varphi_{n} \}$$_{n\geq 1}$ of integrable random
continuous functions on $E$ is called sub-additive if for
${\mathbb{P}}$-almost all $\omega$,
 $$
 \varphi_{n+m}(\omega,x)\leq
\varphi_{n}(\omega,x)+\varphi_{m}(\Theta^{n}(\omega,x))
~{\mathrm{for~all}}~n,m\in \mathbb{N}, x\in E_{\omega}.
$$

Let $\epsilon: \Omega\rightarrow(0,1]$ be a measurable function. A
set $F\subset E_{\omega}$ is said to be
$(\omega,\epsilon,n)$-separated for $T$, if $x,y\in F,x\neq y$
implies $y\notin B_{\omega}(n,x,\epsilon)$, where
$B_{\omega}(n,x,\epsilon):=\{y\in E_{\omega} : d(T(k,\om) x,
T(k,\om) y)< \epsilon (\vartheta^k \omega) ~for~0\leq k\leq n-1
\}$ and $d$ is the given metric on $M$.

Let $\Phi$ = $\{ \varphi_{n} \}$$_{n\geq 1}$ be a sub-additive
function sequence with $\varphi_{n}\in L_{E}^{1}(\Omega,C(M)) $
for each $n$. As usual, we  put
$$
\begin{array}{l}
 {\pi_{T}(\Phi)}(\omega,\epsilon,n)= \sup \{ \sum\limits_{x\in
F} e^{  \varphi_n(\omega,x)}: F \mbox{ is  an }
(\omega,\epsilon,n)\mbox{-separated subset of } E_{\omega}\}\\
{\pi_{T}(\Phi)}(\epsilon)=\limsup\limits_{n\rightarrow
\infty} \frac{1}{n} \int \log{\pi_{T}(\Phi)}(\omega,\epsilon,n)\mathrm{d} \mathbb{P}(\omega)\\
{\pi_{T}(\Phi)}=\lim\limits_{\epsilon
\downarrow0}{\pi_{T}(\Phi)}(\epsilon)
\end{array}
$$
The last quantity  is called the sub-additive topological pressure
of $\Phi$ with respect to $T$. We just mention that the above
definition is reasonable, see \cite{zhao2} for details.

\begin{remark} \label{jiazhu2} \rm
\it (i) If the function sequence $\Phi=\{ \varphi_n\}$ can be
written as $\varphi_n(\omega,
x)=\sum_{i=0}^{n-1}\varphi(\Theta^i(\omega, x))$ for some function
$\varphi\in L_{E}^{1}(\Omega,C(M))$, then  we call
${\pi_{T}(\Phi)}$ the random additive topological pressure, see
\cite{bog,kif1} for details, denote it simply by
${\pi_{T}(\varphi)}$. (ii) Since $\mathbb{P}$ is ergodic in the
model which we consider, so the limits in the above definition
will not change $\mathbb{P}$-almost everywhere without integrating
against $\mathbb{P}$.
\end{remark}

\begin{lemma} \label{yl31} \rm
\it For $i=1,2$, let $X_i$ be compact metric spaces, $E_i$
measurable bundles over $\Om$ with compact fibers in $X_i$, and
$\varphi_i$ topological bundle random dynamical systems on $E_i$.
If $\psi=\{ \psi(\om): E_{\om}^{1}\rightarrow E_{\om}^{2} \}$ is a
family of homeomorphism between $E_{\om}^{1}$  and $E_{\om}^{2}$
satisfying $\varphi_2(\om)\circ \psi(\om)=\psi(\vartheta\om)\circ
\varphi_1(\om),\ \mathbb{P}-a.s.$, and ${\mathcal{F}}=\{
f_n\}_{n\geq 1}$ is a sub-additive potential in $L_{E_2}^{1}(\Om,
C(X_2))$ then
$$
\pi_{\varphi_2}(\mathcal{F})=\pi_{\varphi_1}(\mathcal{F}\circ
\psi)
$$
where $\mathcal{F}\circ \psi =\{ f_n\circ \psi \}_{n\geq 1}$
denotes the member of $L_{E_1}^{1}(\Om, C(X_1))$ defined by
$f_{n}(\om,\psi(\om)x)$ for each $n\geq 1$.
\end{lemma}
\begin{proof}
We first check the new defined potential $\mathcal{F}\circ \psi
=\{ f_n\circ \psi \}_{n\geq 1}$ is indeed sub-additive. Precisely,
we have
$$
\begin{array}{ll}
f_{n+m}\circ \psi (\om,x)&=f_{n+m}(\om,\psi(\om)x) \\
&\leq f_n(\om,\psi(\om)x)+f_m(\vartheta^n\om,\varphi_2(n,\om)\psi(\om)x)\\
&= f_n\circ \psi
(\om,x)+f_m(\vartheta^n\om,\psi(\vartheta^n\om)\varphi_1(n,\om)x)\\
&=f_n\circ \psi (\om,x)+f_m\circ \psi (\vartheta^n\om,
\varphi_1(n,\om)x).
\end{array}
$$

Let $\mu\in \mathcal{M}_{\mathbb{P}}^{1}(E_1,\varphi_1)$ and write
$\mu_{\psi}$ for the member of
$\mathcal{M}_{\mathbb{P}}^{1}(E_2,\varphi_2)$ defined by
$\psi(\om)_{*}\mu_{\om}$. We have
$h_{\mu}^{(r)}(\varphi_1)=h_{\mu_{\psi}}^{(r)}(\varphi_2)$ (see
theorem 2.2.2 in \cite{bog2}) and
$\lim_{n\rightarrow\infty}\frac{1}{n}\int f_n\circ \psi
{\mathrm{d}}\mu=\lim_{n\rightarrow\infty}\frac{1}{n}\int f_n
{\mathrm{d}}\mu_{\psi}$, by the variational principle of random
sub-additive topological pressure in \cite{zhao2} we have
$\pi_{\varphi_1}(\mathcal{F}\circ \psi)\leq
\pi_{\varphi_2}(\mathcal{F})$. By symmetry we get the reverse
inequality and hence the desired result.
\end{proof}

\begin{definition} \label{dy31} \rm
\it Let $T$ be a RDS  over $\vartheta$. A generator of RDS $T$ is
a family ${\mathcal{A}}=\{ {\mathcal{A}}(\omega)=(A_i(\omega)):
{\mathcal{A}}(\omega) \mbox{ is an open cover of } E_\omega \}$
with
\begin{itemize}
  \item[(i)] $\mathcal{A}(\omega)$ is finite for all $\omega\in
  \Omega$;

  \item[(ii)] $\omega\mapsto d(x, A_i(\omega))$ is measurable for all
  $x\in M$ and all $i\in \mathbb{N}$;

  \item[(iiii)] for each sequence $(A_n)_{n\in \mathbb{N}}$ of sets
  we have $A_n \in {\mathcal{A}}(\vartheta^n\omega)$ for all $n\in
  \mathbb{N}$ implies that $\bigcap_{n=0}^{\infty}T(n,\om)^{-1}\bar{A_n}$ contains at
  most one point.
\end{itemize}
\end{definition}

\begin{definition} \label{dy32} \rm
\it Let $T$ be a RDS  over $\vartheta$. We call $T$ is
(positive)expansive if there exists a $(0,1)$-valued random
variable $\Delta$ such that
$$
d(T(n,\om)x, T(n,\om)y)\leq \Delta(\vartheta^n\om)\ for \ all \
n\in\mathbb{N}
$$
implies $x=y$.
\end{definition}

\begin{definition} \label{dy33} \rm
\it A generator ${\mathcal{A}}$ of a given RDS $T$ is called a
strong generator if
$$
\lim_{k\rightarrow\infty} {\mathrm{diam}} \bigvee_{i=0}^{k-1}
T(i,\om)^{-1} {\mathcal{A}}(\vartheta^i\omega)=0~~uniformly\ in\
\omega
$$
An expansive RDS is said to be strongly expansive if it possesses
a strong generator.
\end{definition}

Let $U\subset M$ be an open subset of the Riemannian manifold $M$
with $\bar{U}\in {\mathcal{K}}$ and let $C(U,M)$ denote the space
of all continuous maps from $U$ to $M$ endowed with the compact
open topology.

\begin{definition} \label{dy34} \rm
\it Assume $f\in C(U,M)$ and $J\in \mathcal{K}$ with $fJ=J$. A
family $\{ T_\e \}_{\e >0}$ of $C(U,M)$-valued random variables is
called a random perturbation of $f$ on $J$ if
\begin{itemize}
  \item[(i)] $\lim_{\e\rightarrow 0} T_\e=f$ in probability;

  \item[(ii)] there exists a family of $\mathcal{K}$-valued random
  variables $\{ J_\e \}_{\e >0}$ such that
  \begin{itemize}
    \item[(a)] for each $\e >0$ we have that $\mathbb{P}$-a.s.
    $T_{\e}(\om)J_{\e}(\om)=J_{\e}(\vartheta\om)$;

    \item[(b)] $\lim_{\e\rightarrow 0} J_\e= J$ in probability.
  \end{itemize}
  \end{itemize}
$\{ T_\e\}_{\e>0}$ is said to be structurally stable if there
exists a family $\{ h_{\e}\}_{\e>0}$ of $C(J,M)$-valued random
variables such that
\begin{itemize}
  \item[(iii)]  for each $\e>0$ we have that  $h_{\e}(\om):J\rightarrow J_{\e}(\om)$ is a
  homeomorphism and $T_{\e}(\om)\circ h_{\e}(\om)=h_{\e}(\vartheta\om)\circ
  f$ $\mathbb{P}$-a.s.;

  \item[(iv)] $\lim_{\e\rightarrow 0} h_{\e} =id$ in probability.
\end{itemize}
\end{definition}

See \cite{bog} for examples of strongly expansive bundle RDS and
structurally stable random perturbation. Put $X=\overline{U}$ and
let $C^{\a}(X,\mathbb{R})$ denote the space of all H\"{o}lder
continuous functions on $X$ with H\"{o}lder exponent $\a$. We
endow $C^{\a}(X,\mathbb{R})$ with the usual norm $|| \cdot
||_{\a}:=||\cdot ||+|\cdot|_{\a}$, where $||\cdot||$ is the
sup-norm and $|\cdot|_{\a}$ is the least H\"{o}lder constant,
namely, $| \varphi |_{\a}:=\sup_{x,y\in X,x\neq y}
\frac{|\varphi(x) -\varphi(y)|}{d(x,y)^ \a}$. Then using
proposition \ref{jiamt22} we can get the following important
proposition which is proved in \cite{bog}, we cite here just for
complete.

\begin{proposition} \label{mt31} \rm
\it Let $\{ T_{\e}\}_{\e >0}$ be a structurally stable random
perturbation of $f$ on $J$. And let $\{\varphi_{\e}\}_{\e>0}$ be a
family of $C^{\a}(X,\mathbb{R})$-valued  random variables
satisfying
\begin{eqnarray*}
\lim_{\e\rightarrow 0} || \varphi_{\e}-\phi||_{\a} =0\ in\
L^1(\mathbb{P})
\end{eqnarray*}
for some  $\phi\in C^{\a}(X,\mathbb{R})$. Then
$$
\lim_{\e\rightarrow 0}\pi_{T_{\e}}(\varphi_{\e})=\pi_{f}(\phi).
$$
\end{proposition}

\section{Stability of the Hausdorff dimension under random perturbations}
In this section we will prove that the Hausdorff dimension of the
average conformal repeller is stable under suitable random
perturbation.

The following proposition can be proved by slightly modification
of the proof of theorem 1.1 in \cite{liu}.

\begin{proposition} \label{mt41} \rm
\it Assume $J$ is an average conformal repeller for a $C^{1+\a}$
map $f: U\rightarrow M$. There exists a $C^1$ neighborhood
$\mathcal{U}(f)\subset C^{1+\a}(U,M)$ of $f$ such that the
following holds:

\noindent(i) For every random variable $T:\Om \rightarrow
\mathcal{U}(f)$ there exists a $\mathcal{K}$-valued random
variable $J(\om)\subset U$ satisfying  $T(\om)J(\om)
=J(\vartheta\om) $, and a $C^0(U,M)$-valued random variable $h$
such that each $h(\om)$ is a homeomorphism between $J$ and
$J(\om)$ and $T(\om)\circ h(\om)=h(\vartheta\om)\circ f$ on $J$.

\noindent (ii) If $\{ T_{\e} :\Om \rightarrow \mathcal{U}(f)
\}_{\e>0}$ is a family of random variables with
$\lim_{\e\rightarrow 0} T_{\e} =f$ in probability(with respect to
the $C^1$ distance), then $\lim_{\e\rightarrow 0} h_{\e}
=\mathrm{id}$ in probability and thus $\lim_{\e\rightarrow 0} J_\e
=J$ in probability. Here $h_\e$ and $J_\e$ are the corresponding
objects associated to $T_\e$ by (i).\\
In other words, each family $\{ T_\e \}_{\e >0}$ of
$\mathcal{U}(f)$-valued random variables with $\lim_{\e\rightarrow
0} T_\e=f$ in probability is a structurally stable random
perturbation of $f$ on $J$.
\end{proposition}

Now we state and prove our main result.

\begin{theorem} \label{dl41} \rm
\it Assume $J$ is an average conformal repeller for a $C^{1+\a}$
map $f: U\rightarrow M$. There exists a $C^1$ neighborhood
$\mathcal{U}(f)\subset C^{1+\a}(U,M)$ of $f$ such that the
assertions of proposition \ref{mt41} hold with the following
additional property.

If $\{ T_\e : \Om \rightarrow \mathcal{U}(f)\}_{\e
>0} $ is a random perturbation of $f$ with
$$
\lim_{\e \rightarrow 0} T_\e =f\ in\ L^1(\Om,C^{1+\a} (U,M))
$$
then
$$
\lim_{\e\rightarrow 0} \dim_{H}(J_{\e}(\om))= \dim_{H}(J)\quad
\mathbb{P}\mbox{-a.s.}
$$
where $\dim_{H}(\cdot)$ denote the Hausdorff dimension of a set.
Moreover, if $L\subset J$ is compact and $f$-invariant, then
$\lim_{\e\rightarrow 0} \dim_{H}(h_\e(\om)L)=\dim_{H}(L)\
\mathbb{P}$-a.s.
\end{theorem}
\begin{proof}
 In the following we will follow Bogensch\"{u}tz and Ochs' proof \cite{bog} to obtain the desired result.
  Choose an open neighborhood $V$ of $J$ such that $\overline{V}$
 is a compact subset of $U$. Then $\mathcal{U}(f)$ can be chosen
 in such a way that $J_{\e}(\om)\subset \overline{V}$ for every $\e
 >0, \om \in \Om$.

 Fix $\e >0$. For $(\om,x)\in \Om\times \overline{V}$ set
 \begin{eqnarray*}
 \eta_{\e}(\om,x):=|| D_xT_{\e}(\om)|| \ and \
 \lambda_{\e}(\om,x):= m(D_xT_{\e}(\om)).
 \end{eqnarray*}
 By taking appropriate $\mathcal{U}(f)$, we can assume that $\log \lambda_{\e},\log \eta_{\e}\in
 L_{J_{\e}}^{1}(\Om,C(M))$.

 For the clarity of the proof, we divide the proof into several
 steps.

\textbf{Step 1}: We claim that $T_{\e}$ satisfies the following
formula
\begin{equation}\label{ds41}
\lambda_{\e}(\om,x)-K(\om)d(x,y)^\alpha \leq
\frac{d(T_{\e}(\om)x,T_{\e}(\om)y)}{d(x,y)} \leq
\eta_{\e}(\om,x)+K(\om)d(x,y)^\alpha
\end{equation}
 for every $\om\in\Om$ and $x\neq y\in J_{\e}(\om)$, where $K:\Om \rightarrow
 \mathbb{R}_{+}$ with $\log^{+}K\in L^{1}(\mathbb{P})$. We will
 prove the inequality (\ref{ds41}) in the rest of this step.

 Choose $r_0>0$ such that $A:=\{ x: dist(x,\overline{V})\leq r_0\} \subset
 U$. Define
 $$
 K_0(\om)= |DT_{\e}(\om)|_{\a,A}=\sup \left\{ \frac{\| D_xT_{\e}(\om)
 -D_yT_{\e}(\om)\|}{d(x,y)^{\a}}:x,y\in A, x\neq y\right\}.
 $$
 For the simplicity of notations, we restrict $M$ to be the case
 of an open subset of $\mathbb{R}^{d}$, since the general case can
 be done via local coordinates. We let $|\cdot|$ denote the
 Euclidian norm on $\mathbb{R}^{d}$ and write $T$ instead of
 $T_{\e}(\om)$ for convenience.

 For $x,y\in J_{\e}(\om)$ with $0<|x-y|<r_0$, we put
 $e:=\frac{y-x}{|y-x|}$ and get that
 \begin{eqnarray*}
 |T(x)-T(y)|&=&\left|\int_{0}^{|y-x|}D_{x+te}T (e)\ {\mathrm{d}}t\right|\leq
 \int_{0}^{|y-x|} \| D_{x+te}T\|{\mathrm{d}}t\\
 &\leq&|y-x|\sup \{ \| D_{x+z}T\|:|z|\leq |y-x|\}\\
 &\leq& |y-x| (\| D_xT\|+K_0(\om)|y-x|^{\a}).
 \end{eqnarray*}
 Thus, we get that
 $$
 \frac{|T(x)-T(y)|}{|y-x|}\leq
 \eta_{\e}(\om,x)+K_0(\om)|y-x|^{\a}.
 $$
On the other hand, we can get that
$$
 \begin{array}{ll}
 \displaystyle\frac{1}{|y-x|}& \left|\int_{0}^{|y-x|}D_{x+te}T (e)\ {\mathrm{d}}t\right| \\
 &\geq\inf  \{ |Ae|: A\in \mbox{convex hull of } D_{x+te}T,0\leq t\leq |y-x|\}\\
 &\geq |D_xT (e)| -\sup  \{ |Ae|: A\in \mbox{convex hull of } (D_{x+te}T-D_xT),\\
&\quad0\leq t\leq |y-x|\}\\
&\geq \lambda_{\e}(\om,x)-K_0(\om)|y-x|^{\a}.
 \end{array}
$$
 The last inequality follows from the definition of $\lambda_{\e}$
 immediately.

 Put
 $$
 K(\om)=\max \left\{ K_{0}(\om), \frac{\mathrm{diam}\overline{V}}{r_0},
 \frac{\max \{ ||
 D_xT_{\e}(\om)||:x\in\overline{V}\}}{r_{0}^{\a}}\right\}
 $$
 and then the  inequality (\ref{ds41}) immediately follows.

\textbf{Step 2}: We claim that
 $$
 \int \log \Lambda_{n}^{\e} {\mathrm{d}}{\mathbb{P}}>0
 $$
for some $n\geq 1$, where $\Lambda_{n}^{\e}(\om)=\min_{x\in
J_{\e}(\om)}\prod_{k=0}^{n-1}\lambda_{\e}(\vartheta^k\om,T_{\e}(k,\om)x)$.

Recall that $J$ is an average conformal repeller for $f$. It is
easy to see that $\lim\limits_{\e\rightarrow 0} \log
\Lambda_{n}^{\e}=\min_{x\in J} \{\log m(D_{f^{n-1}x}f)+\cdots+\log
m(D_xf)\}>0$ in probability. By making $\mathcal{U}(f)$ smaller if
necessary we have that $|\log(|D_xT e|)|$ is uniformly bounded for
all $T\in \mathcal{U}(f),x\in \overline{V}$, and $e\in T_xM$ with
$|e|=1$. Thus $\lim\limits_{\e\rightarrow 0} \log
\Lambda_{n}^{\e}=\min\limits_{x\in J} \{\log
m(D_{f^{n-1}x}f)+\cdots+\log m(D_xf)\}>0$ also in
$L^1(\mathbb{P})$, which implies $\sup_{n\geq 1} \frac{1}{n} \int
\log \Lambda_{n}^{\e} {\mathrm{d}}{\mathbb{P}}>0$ for sufficiently
small $\e$. This finishes the proof of the claim.

\textbf{Step 3}: We claim that $T_\e$ is strongly expansive. By
remark \ref{zhu1} we know $f$ is expanding on $J$, then  there
exists a neighborhood $V$ of $J$, a constant $c>0$, and an integer
$n\geq 1$ such that $|D_xf^n (e)|\geq 1+c$ for every $x\in V$ and
$e\in T_xM$ with $|e|=1$. We can choose $\mathcal{U}(f)$ in such a
way that $|D_x(T_n\circ\cdots \circ T_1) (e)|\geq 1+\frac{c}{2}$
whenever $T_1, \ldots, T_n\in \mathcal{U}(f),x\in V$, and $e\in
T_xM$ with $|e|=1$, and that $J_{\e}(\om)\subset V$ for every $\e
>0$ and $\om \in \Om$. Then $T_\e$ is uniformly expanding and thus
strongly expanding on the bundle $\{ J_\e (\om)\}_{\om\in\Om}$.

\textbf{Step 4}: Let $L\subset J$ be a compact subset with $fL=L$.
We apply corollary 3.5 in \cite{bog} to the bundle RDS $T_\e$ on
$J_\e=\{ (\om,x):x\in h_{\e}(\om)L\}$. Let $\pi_\e$ denote the
pressure functional of $T_\e$ restricted to $J_\e$, then we can
get that there exist $s_{1}^{\e}\geq t_{1}^{\e}\geq 0$ such that
$$
\pi_{\e}(-t_{1}^{\e}\log \eta_{\e})=0=\pi_{\e}(-s_{1}^{\e}\log
\lambda_{\e})
$$
and
$$
 t_{1}^{\e}\leq \dim_{H}(h_{\e}(\om)L)\leq
 s_{1}^{\e}\quad\mathbb{P}\mbox{-a.s.}
$$
If we consider the system $T_{\e}(n,\om)$ and $\log \|
D_xT_{\e}(n,\om)\|,\   \log m(D_xT_{\e}(n,\om))$ for every $n>0$,
and let $\pi_{n,\e}$ denote the pressure functional of
$T_\e(n,\om)$ restricted to $J_\e$, then  we can get that
$$
 t_{n}^{\e}\leq \dim_{H}(h_{\e}(\om)L)\leq
 s_{n}^{\e} \quad \mathbb{P}\mbox{-a.s.}
$$
where $t_{n}^{\e},s_{n}^{\e}$ satisfying
$\pi_{n,\e}(-t_{n}^{\e}\log \| D_xT_\e
(n,\om)\|)=0=\pi_{n,\e}(-s_{n}^{\e}\log m(D_xT_\e(n,\om)))$.
 Furthermore, by theorem \ref{dl22} the Hausdorff dimension
of $L$ is the unique $t_0\geq 0$ with $\pi_{f|_{L}}(-t_0\Psi)=0$,
where $\Psi=\{ \log m(D_xf^n): x\in L, n\in \mathbb{N}\}$.

\textbf{Step 5}: We first note that, for each fixed positive
integer $n$, we have
\begin{eqnarray*}
\lim_{\e\rightarrow 0} \| \log \| D_xT_{\e}(n,\om)\|-\log \|
D_xf^n\| \|_{\a}=0
\end{eqnarray*}
and
\begin{eqnarray*}
\lim_{\e\rightarrow 0}\|\log m(D_xT_{\e}(n,\om))-\log
m(D_xf^n)\|_{\a}=0\ in \ L^1(\mathbb{P}),
\end{eqnarray*}
since $|\log \| D_xT_{\e}(n,\om)^{\pm}\||$ is uniformly bounded
for the fixed positive integer $n$. By the  proposition
\ref{mt31}, for each fixed positive integer $n$,  we have
$$
\lim_{\e\rightarrow 0}\frac{1}{n}\pi_{n,\e}(-t\log \|
D_xT_\e(n,\om)\| )=\frac{1}{n}\pi_{f^n|_{L}}(-t\log \| D_xf^n \|)
$$
and
$$
\lim_{\e\rightarrow 0}\frac{1}{n}\pi_{n,\e}(-t\log
m(D_xT_\e(n,\om)) )=\frac{1}{n}\pi_{f^n|_{L}}(-t\log m(D_xf^n))
$$
for each  $t\geq 0$. Moreover, by proposition 2.2 in \cite{cao1}
and proposition \ref{jiamt23}  we have
 $$
 \lim_{n\rightarrow\infty}\frac{1}{n}\pi_{f^n|_{L}}(-t\log
\| D_xf^n
\|)=\lim_{n\rightarrow\infty}\frac{1}{n}\pi_{f^n|_{L}}(-t\log
m(D_xf^n))=\pi_{f|_L}(-t\Psi).
 $$
 Hence, we obtain for each $t\geq 0$ that
 $$
 \begin{array}{ll}
 \displaystyle\lim_{n\rightarrow\infty}\lim_{\e\rightarrow 0}\frac{1}{n}\pi_{n,\e}(-t\log \|
D_xT_\e(n,\om)\|
)&=\displaystyle\lim_{n\rightarrow\infty}\lim_{\e\rightarrow
0}\frac{1}{n}\pi_{n,\e}(-t\log m(D_xT_\e(n,\om)) ) \\
&=\pi_{f|_L}(-t\Psi).
 \end{array}
 $$

\textbf{Step 6}: To complete the proof, given $\delta >0$. Since
$t\mapsto \pi_{f|_L}(-t\Psi)$ is strictly decreasing, there exist
$N>0, \e_0>0$ such that for each $ \e\leq \e_0$, we have
$$
\pi_{N,\e}(-(t_0+\delta)\log \| D_x
T_\e(N,\om)\|)<0<\pi_{N,\e}(-(t_0-\delta)\log \| D_x
T_\e(N,\om)\|)
$$
and
$$
\pi_{N,\e}(-(t_0+\delta)\log m( D_x
T_\e(N,\om)))<0<\pi_{N,\e}(-(t_0-\delta)\log m( D_x T_\e(N,\om))).
$$
This immediately implies
\begin{eqnarray} \label{jiads44}
t_0-\delta <t_{N}^{\e}\leq \dim_{H} (h_{\e}(\om)L)\leq
s_{N}^{\e}<t_0+\delta.
\end{eqnarray}
The desired result then immediately follows.
\end{proof}

\begin{remark} \label{zhu4} \rm
\it (1) In \cite{bog}, Bogensch\"{u}tz and Ochs proved that the
Hausdorff dimension of a conformal repeller is stable under random
perturbations. Using  their ideas, we show that the same is true
for average conformal repeller. The differences between theorem
\ref{dl41} and Bogensch\"{u}tz and Ochs's theorem   are:

i) In order to use the corollary 3.5 in \cite{bog}, it is the same
from step 1 to step 3;

ii) In order to prove the Hausdorff dimension of average conformal
repeller is stable under random perturbation, we should consider
the iteration of the RDS from step 4 to step 6. And this process
need the technic of sub-additive topological pressure and
sup-additive topological
 pressure. In \cite{bog}, the authors need not consider the
 iteration of the RDS, so they need only additive topological
 pressure.

 (2) Since the bundle $T_\e$ is uniformly expanding on the bundle $\{ J_\e(\w)\}_{\w\in\Omega}$, the result
in \cite{liu} told us that there exists a  equilibrium states of
the  topological pressure $\pi_\e$. Then modifying subtly the
proof in \cite{ban} we can get that the zero of the sub-additive
topological pressure is the upper bound of the Hausdorff dimension
of the bundle $\{ J_\e(\w)\}_{\w\in\Omega}$.
\end{remark}

\begin{proposition} \label{mt42} \rm
\it Under the conditions of theorem \ref{dl41}, we have
$$
\lim_{\e\rightarrow 0}h_{top}^{(r)}(T_\e)=h_{top}(f),
$$
where $h_{top}^{(r)}(T_\e)$ denote the  topological entropy of the
random dynamical system $T_{\e}$ generated by the random
perturbation of $f$ and $h_{top}(f)$ denote the classical
topological entropy of deterministic dynamical system.
\end{proposition}
\begin{proof}
This can be immediately deduced from proposition \ref{mt31} by
taking the potential functions to be the zero-valued functions.
\end{proof}

\noindent {\bf Acknowledgements.}Part of this work is done when
authors visited Taida institute for mathematical sciences, authors
would like thank the warm hospitality of the host.

  \end{document}